\documentclass[12pt]{amsart}
\usepackage{amsfonts,graphics,amsmath,amsthm,amsfonts,amscd, amssymb,amsmath,latexsym,enumitem}
\usepackage[left=1.02in,top=1.0in,right=1.02in,bottom=1.0in]{geometry}
\usepackage[all]{xy}
\usepackage{color}
\newcommand{\bQ}{\mathbb{Q}}

\newcommand{\sF}{\mathcal{F}}
\newcommand{\sG}{\mathcal{G}}
\newcommand{\sH}{\mathcal{H}}

\newcommand{\sO}{\mathcal{O}}

\newcommand{\sT}{\mathcal{T}}

\newcommand{\sHom}{Hom}

\DeclareMathOperator{\reg}{reg}
\DeclareMathOperator{\rk}{rk}

\DeclareMathOperator{\im}{Im}

\DeclareMathOperator{\Supp}{Supp}
\DeclareMathOperator{\Ext}{Ext}
\DeclareMathOperator{\sstab}{ss}
\DeclareMathOperator{\nilp}{nilp}

\theoremstyle{plain}
\newtheorem{theorem}{Theorem}[section]
\newtheorem{corollary}[theorem]{Corollary}
\newtheorem{lemma}[theorem]{Lemma}

\newtheorem{notation}[theorem]{Notation}

\newtheorem{proposition}[theorem]{Proposition}
\newtheorem{definition-lemma}[theorem]{Definition-Lemma}

\newtheorem{definition}[theorem]{Definition}

\theoremstyle{definition}

\newtheorem*{acknowledgement}{Acknowledgments}

\begin{document}
\title{Birational characterization of abelian varieties and ordinary abelian varieties in characteristic $p>0$}
\author{Christopher D. Hacon}
\address{Department of Mathematics \\
University of Utah\\
Salt Lake City, UT 84112, USA}
\email{hacon@math.utah.edu}
\author{Zsolt Patakfalvi}
\address{
EPFL\\
SB MATHGEOM CAG  \\
MA B3 444 (B\^atiment MA) \\
Station 8 \\
CH-1015 Lausanne}
\email{zsolt.patakfalvi@epfl.ch}
\author{Lei Zhang}
\address{School of Mathematics and Information Sciences \\
Shaanxi Normal University\\
Xi'an 710062, P. R. China}
\email{zhanglei2011@snnu.edu.cn}

 \thanks{The first author
 was supported by NSF research grants no: DMS-1300750, DMS-1265285 and by
  a grant from the Simons Foundation; Award Number: 256202. The thired author was supported by grant
NSFC (No. 11401358 and No. 11531009).}
\maketitle
\begin{abstract} Let $k$ be an algebraically closed field of characteristic $p>0$.  We give a birational characterization of ordinary abelian varieties over $k$: a smooth projective variety $X$ is birational to an ordinary abelian variety if and only if $\kappa_S(X)=0$  and $b_1(X)=2 \dim X$.
%
We also give a similar characterization of abelian varieties as well:  a smooth projective variety $X$ is birational to an abelian variety if and only if $\kappa(X)=0$, and the Albanese morphism $a: X \to A$ is generically finite.
Along the way, we also show that if $\kappa _S (X)=0$ (or if $\kappa(X)=0$ and $a$  is generically finite) then the Albanese morphism $a:X\to A$ is surjective and in particular $\dim A\leq \dim X$. \end{abstract}

Let $X$ be a smooth projective variety over an algebraically closed field $k$, then one hopes to classify $X$ 
via its birational invariants such as   the Betti numbers and the plurigenera $P_m(X)=h^0(\omega _X^{\otimes m})$. Classification results of this kind are quite rare, however over fields of characteristic $0$ there have been many striking results in the case of irregular varieties i.e. varieties with $q(X):=h^0(\Omega ^1_X)>0$. A famous classical example is the celebrated result of Kawamata in \cite{Kawamata81}, saying that a smooth projective variety $X$ over an algebraically closed field of characteristic $0$ is birational to an abelian variety if and only if $b_1(X)= 2 \dim X$ and  $\kappa(X)=0$ (or equivalently if $P_m(X)=1$ for all $m>0$ sufficiently divisible).

We show a positive characteristic version of this theorem for ordinary abelian varieties. Suppose from now on  that the base field $k$ is algebraically closed and of characteristic $p>0$. Recall that an abelian variety $A$ over $k$ is ordinary if and only if the action of the Frobenius on $H^1(A, \sO_A)$ is bijective. These are general abelian varieties in the sense that their locus is dense and open in the moduli space of abelian varieties. We prove the following.

\begin{theorem}
\label{thm:char_ord}
A smooth projective variety $X$ over an algebraically closed field of characteristic $p>0$ is birational to an ordinary abelian variety if and only if $\kappa_S(X)=0$ and $b_1(X) = 2 \dim X$.
\end{theorem}

Here $\kappa_S(X)$ denotes the Frobenius stable Kodaira dimension (the growth rate of the dimension of the Frobenius stable pluri-canonical forms, see \cite[Section 4.1]{HP13} for the precise definition). It is easy to see that an abelian variety $A$ is ordinary if and only if $\kappa _S(A)=0$ \cite[2.3.2]{HP13}.  We have that $b_1(X):=\dim_{\bQ_l}H^1_{\textrm{\'et}}(X, \bQ_l)$, which is known to be equal to $\dim A$ \cite[page 14]{L09}.

As usual in characteristic $p>0$, the main difficulties in proving Theorem \ref{thm:char_ord} correspond to the case when the Albanese morphism $a: X \to A$ is wildly ramified or inseparable (see examples of the latter in \cite[8.7]{L09}). In fact, the first two authors proved the above statement excluding roughly these cases in \cite[Thm 1.1.1]{HP13} (to be precise the condition was that $p \nmid \deg a$). So, the main novelty of Theorem \ref{thm:char_ord} is the removal of this additional  assumption on $a$.

We are also able to prove the following result towards the birational classification of non-ordinary abelian varieties.

\begin{theorem}
\label{thm:char_ab}
A smooth projective variety $X$ over an algebraically closed field of characteristic $p>0$ is birational to an abelian variety if and only if $\kappa(X)=0$ and the Albanese morphism $a :X \to A$ is generically finite.
\end{theorem}

The main reason for the appearance of the generically finite assumption on $a$, as opposed to the Betti number assumption of Theorem \ref{thm:char_ord}, is that when $\kappa(X)=0$ but $\kappa_S(X) = - \infty$ we are unable to construct a useful  non nilpotent Cartier module on $A$ given by the geometry of $a : X \to A$. We do not know if it is possible to strengthen Theorem \ref{thm:char_ab} to a cohomological statement similar to Theorem \ref{thm:char_ord}.

Although in proving the above results we use very different methods than those in \cite{Kawamata81}, the initial reduction is the same. That is,  Kawamata in  \cite{Kawamata81} shows first that the Albanese map $a:X\to A$ is surjective (and in particular $\dim A\leq \dim X$), and then he proceeds by showing that if $\dim A=\dim X$ holds as well then in fact $a : X\to A$ is birational.
We follow the same structure, and hence the precise theorems shown in this article are as follows (see Theorem \ref{t-surj},   Theorem \ref{t-K} and Corollary \ref{cor:ordinary} for the actual proofs):

\begin{theorem}
\label{thm:precise} Let $X$ be a smooth projective variety defined over an algebraically closed field of characteristic $p>0$ and $a:X\to A$ its Albanese morphism, then
\begin{enumerate}
\item
If $\kappa (X)=0$ and the Albanese morphism is generically finite, then $a$ is surjective
(and in particular $\dim A\leq \dim X$).
\item
If $\kappa_S(X)=0$ then $a$ is surjective.
\item If $\kappa (X)=0$,  $\dim A=\dim X$,  and $a$ is generically finite
then $a:X\to A$ is birational.
\end{enumerate}
\end{theorem}

Note that Theorem \ref{thm:precise}, with the additional statement of Corollary \ref{cor:ordinary}, indeed implies  Theorem \ref{thm:char_ord} and  Theorem \ref{thm:char_ab}.
\subsection{Outline of the proof} In order to prove these result, we use several refinements of the generic vanishing techniques developed in \cite{HP13}.

First, we sketch the ideas going into the surjectivity of $a$  for the case $\kappa_S(X)=0$ (the $\kappa(X)=0$ case is very similar).
The main idea is to consider the morphism $F_*a_*\omega _X\to a_*\omega _X$ induced by the trace of the Frobenius on $X$.
This induces an inverse system $F^{e+1}_*a_*\omega _X\to F^e_*a_*\omega _X$ that satisfies the Mittag Leffler condition.
We let $\Omega =\varprojlim F^e_*a_*\omega _X$.
If $\kappa_S (X)=0$ (and assuming for simplicity that $h^0(K_X)\ne 0$), then in particular one has that $H^0(\Omega )\ne 0$ and so $\Omega \ne 0$.
We then study the  Fourier-Mukai transform of the dual direct limit $\Lambda =\varinjlim R\hat S D_A(F^e_*a_*\omega _X)$.
By \cite{HP13}, it follows,  that $\Lambda$ is a quasi-coherent sheaf on $\hat A$ such that $(-1_A)^*D_A(RS(\Lambda ))[-g]=\Omega$. It turns out that the support of $\Lambda _0'$, the coherent sheaf on $\hat A$ given by the image of $\mathcal H ^0(R\hat S D_A(a_*\omega _X))\to \Lambda$,  is contained in the cohomological support locus $V^0(K_X)=\{P\in \hat A|h^0(K_X-P)\ne 0\}$ and is invariant via multiplication by $p$ (on $\hat A$).  It then follows that
in fact the support of $\Lambda '_0$ equals to $\{\mathcal O _A\}\subset \hat A$ (Lemma \ref{lem:support}). But then the Fourier-Mukai transform $RS(D_A(\Lambda '_0))$ is a
unipotent vector bundle which we denote by $V_0$. This means that $V_0$ is obtained by successive non-split extensions of $\mathcal O _A$. 
Since $\Omega =\varprojlim F^e_* V_0$  and  $F^e_*V_0 \to F^{e-1}_* V_0$ is surjective, it furthermore follows that $\Supp \Omega = A$, and hence $a$ has to be surjective.

By the surjectivity part, in either case we may assume that $\kappa(X)=0$, and $a$ is surjective and generically finite. Then, under these assumptions we show that $a$ is birational. Similarly as above, for simplicity we assume that $H^0(X,\omega_X) \neq 0$.
Replacing $A$ and $X$ by appropriate covers, we may assume that $V_0=\oplus \mathcal O _A$ (see Lemma \ref{l-1}). Note that this is a phenomenon that fails in characteristic $0$, but  holds in characteristic $p>0$ because in the latter case one can always kill cohomology by passing to finite covers, which can be chosen to be an  isogeny in the present case. Since $V_0\to a_*\omega _X$ is generically surjective and $h^0(a_*\omega _X)=1$, it follows easily that the image of  $V_0\to a_*\omega _X$ has generic rank 1 and hence the generic rank of $a_* \omega _X$ is 1.  Thus $a$ has generic degree 1 and hence is birational. The main subtlety in this argument is that when $X$ is replaced by the pullback via the above isogeny, it has to stay integral. This would not be clear if $a$ was not separable. However, luckily a result of Igusa (stated explicitly in Serre's article \cite{S58}) can be used to show that $a$ is separable if it is surjective, generically finite, and  $\kappa(X)=0$ (Proposition \ref{prop:Albanese_separable}).

\begin{acknowledgement}
Part of this work was done while the third author visited Xiamen University and the Institute for Mathematical Sciences of the National University of Singapore. He would like to thank Prof. Wenfei Liu and De-Qi Zhang for their hospitality and support.
\end{acknowledgement}

\section{Preliminaries}
Throughout this paper we work over an algebraically closed field $k$ of characteristic $p>0$.
We use the notation established in \cite{HP13}. In particular ordinary abelian varieties and their properties are discussed in \cite[\S 2.3]{HP13}, the Fourier-Mukai functor is  discussed in \cite[\S 2.4]{HP13}, the Frobenius stable Kodaira dimension $\kappa _S(X)$ and its relation to the usual Kodaira dimension $\kappa (X)$ is  discussed in \cite[\S 4.1]{HP13}.
\subsection{$p$-closed subsets}
\begin{lemma}
\label{lem:p_closed}
Let $V\subset A$ be a closed subset of an abelian variety
such that $pV :=\{ pv | v \in V \} \subseteq  V$.
If $V'$ is an irreducible component of $V$ of maximal dimension and  $A'$ is  the  abelian subvariety of $A$ spanned by a translate through the origin of $V'$, then there is an integer $m>0$ and a
closed point $\alpha \in A$ such that
\begin{enumerate}
\item $m\alpha =0$, and
\item $V'\subset \alpha +A'$.\end{enumerate}
\end{lemma}
\begin{proof} By definition $A'$ is the smallest abelian subvariety of $A$ containing a translate through the origin of $V'$.
Note that $A'$ does not depend on which translate of $V'$ through the origin we are considering. Indeed, if $P,Q \in V'$, and $A_P$ and $A_Q$ are the abelian subvarieties spanned by $V'-P$ and $V'-Q$, respectively, then $Q-P \in V'-P \subseteq A_P$. Therefore,  $V'-Q = V'-P- (Q-P) \subseteq A_P - (Q-P) = A_P$, which then implies that $A_Q \subseteq A_P$. By symmetry, $A_P \subseteq A_Q$ also holds and this implies that $A_P = A_Q$.

Since multiplication by $p$ is a finite morphism, $pV'$ is also a component of $V$ of maximal dimension.
Furthermore, since $V$ has finitely many components, we have $p^kV'=p^jV'$ for some integers $j>k>0$. Therefore, for any $P \in V'$,
\begin{equation*}
p^jP\in p^jV'=p^kV'\subset p^kP+A',
\end{equation*}
and hence $(p^j-p^k)P \in A'$. Since $A'$ is $p^j-p^k$ divisible, we may find $Q'\in A'$ such that
$(p^j-p^k)(P+Q')=0$. Let $\alpha :=P+Q'$,  then $V'\subset P+A'=\alpha +A'$.\end{proof}

\subsection{Homogeneous vector bundles}

\begin{definition}
\label{def:unipotent}
A vector bundle $V$ on an abelian variety $A$ is {\bf homogeneous} if it admits a filtration $V=F_r\supset F_{r-1}\supset \ldots \supset F_1\supset F_0=0$ such that $F_i/F_{i+1}\cong P_i\in {\rm Pic}^0(A)$ for $1\leq i\leq r={\rm rk}(V)$. $V$ is {\bf unipotent}, if $P_i \cong \sO_A$ for all $i$.
\end{definition}
Note that if $P\in {\rm Pic}^0(A)$ corresponds to the point $y\in \hat A$, then the Fourier-Mukai transform of $P$ is $R\hat S(P)=R^g\hat S (P)[-g]=k_{-y}[-g]$ where $g=\dim A$, $k_{-y}$ denotes the skyscraper sheaf supported at $-y\in \hat A$ and $[-g]$ denotes shifting a complex $g$ spaces to the right. It is then easy to see that the Fourier-Mukai transform induces an equivalence between the category of homogeneous vector bundles on $A$ and that of Artinian $\mathcal O _{\hat A}$-modules.

\begin{lemma}\label{l-1} Let $V$ be a homogeneous vector bundle on $A$ such that all the factors are isomorphic to  $P \in {\rm Pic}^0(A)$. Then there exists an isogeny $\alpha :A'\to A$ such that $\alpha ^*V= \alpha^* P^{\oplus \rk V}$. Furthermore, if $A$ is ordinary, then $\alpha$ can be chosen to be \'etale.
\end{lemma}
\begin{proof}
Let $0 = F_0 \subseteq F_1 \subseteq \dots \subseteq F_{\rk V}$ be the filtration given by Definition \ref{def:unipotent}.

\emph{We claim that it is enough to show that for any abelian variety $B$, there is an isogeny (resp. separable isogeny in the ordinary case)  $\tau_B : B' \to B$, such that $\tau^* : H^1(B,\sO_B) \to H^1(B',\sO_{B'})$ is the zero map.} Indeed,  we show by induction on $i$ that there is an isogeny (resp. separable isogeny) $\alpha_i : A_i \to A$, such that $\alpha_i^* F_i$ splits, that is, $\alpha_i^* F_i \cong \alpha_i^* P^{\oplus i}$. This is a tautology for $i=1$. Then, for the induction step, we want to show that for $\alpha_{i+1} :=  \alpha_i \circ \tau_{A_i}$, the following exact sequence splits:
\begin{equation*}
\xymatrix{
0 \ar[r] & \alpha_{i+1}^* F_i \ar[r] & \alpha_{i+1}^* F_{i+1} \ar[r] & \alpha_{i+1}^* F_{i+1}/ \alpha_{i+1}^*F_i \ar[r] & 0.
}
\end{equation*}
Note that in the ordinary case, if $\alpha_i$ was separable, then by the above definition, so is $\alpha_{i+1}$. In any case, the splitting of the above exact sequence, is determined by an element of
\begin{multline*}
\Ext^1_{A_{i+1}}( \alpha_{i+1}^* F_{i+1}/ \alpha_{i+1}^*F_i, \alpha_{i+1}^* F_i)
%
%
\\ \cong \underbrace{H^1(A_{i+1}, \alpha_{i+1}^* (P^\vee \otimes P^{\oplus i})) }_{\tau_{A_i}^* F_i \cong \tau_{A_i}^* (\alpha_i^* P^{\oplus i}) \textrm{ by the induction hypothesis}}
\cong H^1(A_{i+1}, \sO_{A_{i+1}})^{\oplus i}
\end{multline*}
 coming from an element of $\Ext^1_{A_{i}}( \alpha_{i}^* F_{i+1}/ \alpha_{i}^*F_i, \alpha_{i}^* F_i) \cong H^1(A_{i}, \sO_{A_{i}})^{\oplus i}$ via $\tau_{A_i}^*$. By the choice of $\tau_{A_i}$ this element then has to be zero, which concludes the proof of the induction step, and hence also of the claim itself.

We are left to show that one can find $\tau_B$ as above. Note that the action of the absolute Frobenius $F$ on $H^1(B, \sO_B)$ is $p$-linear. Hence by \cite[III Lemma 3.3]{C98} $H^1(B, \sO_B)$ splits into a semi-stable part $H^1(B, \sO_B)_{\sstab}$ and a nilpotent part $H^1(B, \sO_B)_{\nilp}$, and furthermore, $H^1(B, \sO_B)_{\sstab}$ is generated by vectors $v_1,\dots,v_r$ stable under the action of $F$.

Each $v_i$ corresponds to a rank $2$-vector bundle of the form
\begin{equation*}
\xymatrix{
0 \ar[r] & \sO_B \ar[r] & E_i \ar[r] & \sO_B \ar[r] & 0
},
\end{equation*}
for which $F^* E_i \cong E_i$, since $F^* (v_i)=v_i$. Then, by \cite{LS77} there is an \'etale cover $\rho_i : B_i \to B$, such that $\rho_i^* E_i \cong \sO_{B_i}^{\oplus 2}$. This then implies that in fact the above exact sequence becomes split, that is $\rho_i^*(v_i)=0$. Choose then $\tau_{B} : B' \to B$ to be the fiber product of $F^{\rk H^1(B, \sO_B)_{\nilp}} : B \to B$ and  the morphisms $\rho _i:B_i\to B$. The morphism $\tau_{B}$ satisfies our requirements. Note that in the ordinary case, $\rk H^1(B, \sO_B)_{\nilp}=0$, so $B'\to B$ is \'etale.
\end{proof}

\subsection{Inseparability of the Albanese morphism}

In the following proposition we use the notions of $\kappa(D) = - \infty$, $\kappa(D) = 0$ and $\kappa(D) >0$ for a Weil divisor $D$ on a normal variety $X$ (and $\kappa(D) \leq 0$ and $\kappa(D) \geq 0$, which are just obvious combinations of these). We define these three notions respectively as:
\begin{itemize}
 \item $\kappa(D)=-\infty$ if $H^0(X, mD) =0$ for all $m>0$,
 \item $\kappa(D)=0 $ if ${\rm max}\{h^0(X, mD)\  |\ m >0\}=1$, and
 \item $\kappa(D)>0$ if $h^0(X, mD) >1$ for some $m>0$.
\end{itemize}
Also, by $\det$ we mean the determinant as a (linear equivalence class of) divisor(s).

\begin{proposition}
\label{prop:Albanese_separable}
If $X$ is a normal projective variety with $\kappa(K_X) \leq 0$ and the Albanese morphism $a : X \to A$ is    generically finite, then $a$ is separable.
\end{proposition}

\begin{proof} In what follows we use the the notation and results from \cite{E87}.
Let us assume that the above statement is false. Take then a counterexample $X$ such that the degree of  $a$ is minimal. Set $\sF:= \ker (\sT_{X_{\reg}} \to (a|_{X_{\reg}})^* \sT_A)$, where $\sT_{X_{\reg}}:= \sHom_{X_{\reg}}(\Omega_{X_{\reg}},\sO_{X_{\reg}})$ is the sheaf of derivations on $X_{\reg}$. Then $\sF$ is the sheaf of (local) derivations on ${X_{\reg}}$  that vanish when restricted to pullbacks of (local) functions on $A$. In particular, this property is closed under $p$-th power and Lie-bracket. Furthermore, as the image of $\sT_{X_{\reg}} \to (a|_{X_{\reg}})^* \sT_A$ is torsion-free, $\sF$ is saturated. Hence $\sF$ is a foliation. Let  $Ann (\sF):=\{s\in \sO_X| \xi (s|_{X_{\rm reg}})=0,\ \forall \xi \in \sF\}$ and  $Y:={\rm Spec }_X Ann(\sF)$, which is known to be normal (e.g. \cite[line 6 of page 105]{E88} claims this over $X_{\reg}$ and then it is easy to see that $Ann(\sF)$ is $S_2$, as $s \in \sO_X$ is contained in $Ann(\sF)$ if and only if so is $s|_{X_{\reg}}$).

Since $\sF$ is trivial on pullbacks of functions from $A$, $Ann(\sF)$ contains $f^{-1} \sO_A \subseteq \sO_X$. This then implies that
 there is is an induced morphism $b : Y \to A$ such that $b \circ g =a$, where $g : X \to Y$ is the induced morphism. Indeed, topologically, $b$ agrees with $a$, and then on the ring level, the containment $Ann(\sF)\supseteq f^{-1} \sO_A$ defines it.
Furthermore, for the future reference, as $g$ is finite, we may choose smooth big open sets $Y^0 \subseteq Y$ and $X^0 \subseteq X$, such that $g^{-1}(Y^0)=X^0$ and $\sF|_{X^0}$ is locally free.

Note that $b$ is also an Albanese morhism, since a factorization through a morphism of abelian varieties would yield also a similar factorization of $a$. Assume for a second that we know that $\kappa(K_{ Y}) = - \infty$. Since the degree of  $ b$ is smaller than the degree of $a$, by our initial assumption, it would follow that $ b$ is separable, and hence it would follow that $\kappa(K_{ Y}) \geq 0$, which is a contradiction.

So, we only have to show that $\kappa(K_{ Y} ) = - \infty$. By the canonical bundle formula for quotients by foliations (identifying divisors on $X$ and $X^0$ and on $ Y$ and $Y^0$, and using \cite[Cor 3.4]{E87})  we have $K_{X} + (1-p) \det \sF = f^*K_{ Y}$. Note that if $\sG$ is the saturation of $\im(a^* \Omega_A \to \Omega_{X}$), then
\begin{enumerate}
 \item $\sF \cong (\Omega_{X_{\rm reg}}/\sG|_{X_{\rm reg}})^* \  \Rightarrow \  \det \sF = \det \sG -K_{X}$, and
 \item as by \cite[Th\'eor\`eme 4]{S58} the natural map  $H^0(A, \Omega_A) \to H^0(X, \Omega_X)$ is an injection, and $\sG$ is generically generated by  a $\dim H^0(X,\sG)> \rk \sG$ dimensional section space,  $\kappa(\det \sG)>0$ \cite[Lemma 4.2]{Z16}.
 \end{enumerate}
In particular, it follows that $f^* K_{ Y }= K_X + (1-p)(-K_X + \det \sG) = p K_X + (1-p) \det \sG$. Now, assume that $\kappa(K_{ Y}) \geq 0$. Then, we would have $\kappa(pK_X + (1-p) \det \sG) \geq 0$, and then using that $\kappa(\det \sG)>0$ we would obtain that $\kappa(K_X)>0$, which is a contradiction.
\end{proof}

\section{Cartier modules and their Fourier-Mukai transforms}
We now introduce two Cartier modules $\Omega _0$ corresponding to the cases $\kappa (X)=0$ and $\kappa _S(X)=0$. Understanding the properties of these Cartier modules and their Fourier-Mukai transforms is the key to proving all the results in this paper.
\begin{notation}
\label{notation:Lambda}
For a smooth, projective variety $X$ over $k$, with Albanese morphism $a : X \to A$ and Kodaira dimension $\kappa(X)=0$, we introduce the following two cases of notation:
\begin{enumerate}

\item We have two possibilities (two different cases) for the Cartier modules $\Omega_e$, which will be referenced as notations Notation \ref{notation:Lambda}.$(a)$ and Notation \ref{notation:Lambda}.$(b)$:

\begin{enumerate}
 \item \label{itm:m_equals_1}
$\Omega_e:= F^e_* a_* \omega_X$. We regard $\Omega_0=a_* \omega_X$ as a Cartier module with the structure map $F_* \Omega_0 \to \Omega_0$ induced by the Gorthendieck trace of $X$.
\item \label{itm:arbitrary_m} If $\kappa_S(X)=0$ is also assumed \cite[Sec 4.1]{HP13}, then we sometimes define $\Omega_e:= F^e_* a_* \omega_X^r$, where $r$ is the Calabi-Yau index of $X$, that is the smallest positive integer such that $H^0(X, r K_X) \neq 0$. It is known that $r | p-1$ \cite[Lemma 4.2.4]{HP13}. Let $D \in |rK_X|$ be the unique element. Then multiplication by $\frac{(p^e-1)(r-1)}{r}D$ composed with the trace map yields a Cartier module structure on $\Omega_0$ (see \cite[Lemma 2.2.3 \& Lemma 4.2.5]{HP13}).
\end{enumerate}
\end{enumerate}
In either case we define:
\begin{enumerate}[resume]
\item $\Lambda_e := R \hat S ( D_A( \Omega_e )) = V^{e,*} \Lambda_0 $, where the Cartier module structure on $\Omega_0$ induces natural maps $\Lambda_{e-1} \to \Lambda_e$.
\item $\Lambda:= \varinjlim \sH^0(\Lambda_e)$.
\item $\Lambda_e':= \im ( \Lambda_e \to \Lambda)$. Then $\Lambda_e'= V^{e,*} \Lambda_0'$ and $\varinjlim \Lambda_e' = \Lambda$ hold. Note that there are natural injections $\Lambda_{e-1}' \to \Lambda_e'$, in particular $\Supp \Lambda_0' \subseteq \Supp V^{1,*} \Lambda_0'$ which implies
    $p \Supp \Lambda_0' \subseteq \Supp \Lambda_0'$.
\item Define $V_e:= ((-1_A)^* D_A R S(\Lambda_e'))[-g]$. Note that there are natural surjections $V_e\to V_{e-1}$.
\item $\Omega= \varprojlim \Omega_e$, where $\varprojlim$ is taken in the category of $\mathcal O _A$ modules, as opposed to the category of quasi-coherent sheaves. According to \cite[Theorem 3.1.1]{HP13}, $\Omega= ((-1_A)^* D_A R S(\Lambda))[-g] = \varprojlim V_e $.
\end{enumerate}

\end{notation}

\begin{lemma}
\label{lem:support}
If $X$ is a smooth projective variety with $\kappa  (X)=0$ and $\Omega \neq 0$, then using Notation \ref{notation:Lambda}.$(a)$  $\Supp \Lambda_0'=\{P\}$ for some torsion point $P \in \hat A$ and $V_0\otimes P$ is a
unipotent vector bundle. 
If furthermore, $\kappa_S(X)=0$ is assumed, then using Notation \ref{notation:Lambda}.$(b)$, $\Supp \Lambda_0'=\{0\}$ and $V_0$ is a unipotent vector bundle. 
\end{lemma}
\begin{proof}
We show the two cases at once indicating the differences between the two cases at the adequate places. Let $V:= \Supp \Lambda_0'$, $V'$ a component of $V$ of maximal dimension and $m$, $\alpha$ and $A'$ as in Lemma \ref{lem:p_closed}. 
Note that $V'$ exists (or equivalently $V \neq \emptyset$), since if $\Supp \Lambda_0'=\emptyset$, then $\Lambda=\varinjlim \Lambda_e' =0$, and then $\Omega =0$, which contradicts our assumption.

Since we may freely replace $m$ by any of its multiples, we may assume $m>\dim A'$ and hence $(V- \alpha)^{\times m}=V^{\times m}\to A'$ is surjective. Note also that if $\dim V >0$, then the above surjective map has positive dimensional fibers. In particular, in that case, there are infinitely many $m$-tuples $(P_1, \dots, P_m)\in V^{\times m}$, such that
\begin{equation*}
P_1+ \dots + P_m= 0.
\end{equation*}
 If $\dim V=0$ the equation $P_1 + \dots + P_m = 0$ still holds, but we can guarantee only one $m$-tuple, for which it is satisfied.

At this point we have to consider a multiplication map, although a slightly different one in the two cases of Notation \ref{notation:Lambda}.

In case $(a)$, we consider
\begin{equation*}
H^0(K_X-P_1)\otimes \ldots \otimes H^0(K_X-P_m)\to H^0(mK_X),
\end{equation*}
and note that for all $1 \leq i \leq m$,  $H^0(K_X-P_i) \neq 0$, since:
\begin{equation*}
P_i \in \Supp \Lambda_0' \Rightarrow P_i \in \Supp \Lambda_0
 \underbrace{\Leftrightarrow}_{\textrm{\cite[Cor 3.2.1]{HP13}}} H^0(A, \Omega_0 \otimes P_i^\vee) \neq 0 \Rightarrow H^0(X, \omega_X \otimes a^*P_i^\vee)\ne 0.
\end{equation*}
Similarly, in case $(b)$ consider
\begin{equation*}
H^0(rK_X-P_1)\otimes \ldots \otimes H^0(rK_X-P_m)\to H^0(mrK_X).
\end{equation*}
and note (by a similar argument) that for all $1 \leq i \leq m$,  $H^0(rK_X-P_i) \neq 0$.

In either case, if $\dim V>0$, by the infinite choices of the $m$-tuples $P_1, \dots, P_m$ we obtain infinitely many elements of $|m K_X|$, resp., $|mrK_X|$, and this contradicts the $\kappa(X)=0$ assumption. Therefore, $\dim V=0$ and each point of $V$ is a torsion point (using Lemma \ref{lem:p_closed}).
We claim that in fact $V$ consists of a unique point.
Suppose that $P\ne P'\in V$  and $D\in |rK_X-P|$, $D'\in |rK_X-P'|$ (where $r=1$ in case (a)), then since $P,P'$ are torsion, there is an integer $m>0$ such that $mP=mP'=0$, but then $mD,mD'\in |mrK_X|$ contradicting the $\kappa (X)=0$ assumption.
This concludes the support statement in the $\kappa(X)=0$ case.
The fact  that the support is $\hat 0$ in the  $\kappa_S(X)=0$ case  follows since  $k \cong H^0(A, F^{er}_* a_* \omega_X^r)  \to H^0(A, a_* \omega_X^r) \cong k$ identifies with $\Lambda_0 \otimes k(0) \to \Lambda_e \otimes k(0)$ (by \cite[Cor 3.2.1]{HP13} and Notation \ref{notation:Lambda}.(2)), and the former is a bijective map by the $\kappa_S(X)=0$ assumption.

The statement that $V_0 \otimes P$ is a unipotent vector bundle follows from the isomophisms (using \cite[Sec. 3]{M81})
\begin{equation*}
V_0 \otimes P \cong ((-1_A)^* D_A R S(\Lambda_0'))[-g] \otimes P\cong   R S(D_{\hat A} (\Lambda_0')) \otimes P \cong R S(T_{-P}(D_{\hat A} (\Lambda_0')))
\end{equation*}
as $\Lambda_0'$ is a coherent sheaf supported at $P \in \hat A$, and hence, so is $D_{\hat A} (\Lambda_0')$, and then $T_{-P}(D_{\hat A} (\Lambda_0'))$ is a coherent sheaf supported at $\hat 0 \in \hat A$.


\end{proof}

\begin{lemma}
\label{lem:big_open_set_surjective}
Let us assume, using Notation \ref{notation:Lambda}.$(a)$, that $a$ is generically finite. Then, over the open set $U$ where  $a$ is  finite,  $V_0 \to a_* \omega_X$ is surjective.
\end{lemma}

\begin{proof}
Over $U$, $F^e_* a_* \omega_X \to a_* \omega_X$ is surjective. Hence, $\varprojlim V_e \cong \varprojlim F^e_* a_* \omega_X \to a_* \omega_X$ is surjective. However, this map factors through $V_0 \to a_* \omega_X$, which then also has to be surjective.
\end{proof}

\section{Proofs}

\begin{theorem}\label{t-surj}
If $X$ is a smooth projective variety for which either
\begin{enumerate}
\item $\kappa_S(X)=0$, or
\item $X$ is of maximal Albanese dimension with $\kappa(X)=0$,
\end{enumerate}
then the Albanese morphism $a : X \to A$ is surjective.
\end{theorem}

\begin{proof}
First, we prove the surjectivity statement.
In the case of assumption $(1)$, we use Notation \ref{notation:Lambda}.$(b)$. Then, as the inverse system $\Omega_e$ is Mittag-Leffer \cite[Lem 13.1]{G04}\cite[Prop 8.1.4]{BS13}, $H^0(\Omega) = \varprojlim H^0(\Omega_e) = k$. So, $\Omega \neq 0$. In the case of assumption $(2)$, we use  Notation \ref{notation:Lambda}.$(a)$. Since $a$ is generically finite in this case, $\Omega \to a_*\omega _X$ is generically surjective (Lemma \ref{lem:big_open_set_surjective}) and  in particular $\Omega \ne 0$.

In either case, by the assumptions and Lemma \ref{lem:support}, $\Supp \Lambda_0'$ is a single torsion point. However, then the $V_e$ are homogeneous vector bundles, and $V_{e+1} \to V_e$ are surjective. Since $\Omega =\varprojlim V_e$, it follows that $\Supp \Omega =A$, which then implies surjectivity of $a:X\to A$.

\end{proof}


\begin{theorem}\label{t-K}
Let  $X$ be a smooth projective variety of maximal Albanese dimension with $\kappa(X)=0$.
Then the Albanese morphism $a: X \to A$ is birational. 
\end{theorem}

\begin{proof}
We use Notation \ref{notation:Lambda}.$(a)$ throughout the proof. As in the proof of Theorem \ref{t-surj}, $\Omega \neq 0$.
According to Theorem \ref{t-surj} and Proposition \ref{prop:Albanese_separable}, $a: X \to A$ is surjective and separable.
Let $\tau : B \to C \to A$ be a composition of isogenies such that
\begin{enumerate}
\item $\tau^* V_0$ is trivial,
\item $C \to A$ is \'etale, and
\item $B \to C$ is inseparable.
\end{enumerate}
The existence of such a factorization follows from Lemma \ref{l-1}, Lemma \ref{lem:support} and \cite[Prop 4.45]{MvdG}. Note that, since by Lemma \ref{lem:support} $V_0 \otimes P$ is a unipotent vector bundle for some torsion line bundle $P$, $B \to A$ should dominate one isogeny that trivializes $P$ and also another one that trivializes the 
unipotent vector bundle.  This can be achieved for example by taking an isogeny attaining the latter (using Lemma \ref{l-1}), and then precomposing it with an adequate multiplication by $n$. We set $Z:= C \times_A X$.

\emph{We claim that $Z$ is a smooth projective variety over $k$ with $\kappa(Z)=0$.}
Smoothness follows since $C \to A$, and hence also $Z \to X$, are \'etale. To prove irreducibility note that the morphism $C \to A$ is given by a quotient by a finite group $H$. If $Z$ is not irreducible, then the stabilizer of any of its components is a proper subgroup $H' \subseteq H$. Furthermore, $Z/H'$ is the disjoint union of components isomorphic to $X$. Hence, any of these components induce a factorization $X \to C/H' \to A$. Since $H'$ is a proper subgroup of $H$,  $C/H' \to A$ is a non-trivial isogeny wich contradicts the assumption that $a : X \to A$ is the Albanese morphism. This concludes the proof of irreducibility. As $Z \to X$ is \'etale, $\kappa(Z)=0$ also holds, which concludes the above claim.

Define then $W:=Z \times_C B$. The morphism $ W \to Z$ is a topological isomorphism, so $W$ is irreducible. Furthermore, $W \to B$ is separable, so $W$ is generically reduced. Also, as $B \to C$ is flat with $\omega_{B/C} = 0$, $W$ is Gorenstein (and hence by the generic reducibility also integral) and $\omega_W= f^* \omega_X$, where $f : W \to X$ is the induced morphism. Let $\nu :\tilde W \to W$ be the normalization, then $ \nu ^* K_W$  is a Cartier divisor and so   $\kappa(\nu ^* K_{ W})=0$. However, then it follows that $h^0(b_*\nu _*  \nu ^* \omega_{W}) \leq 1$, where $b : W \to B$ is  the induced morphism. On the other hand, we have a generically surjective morphism (according to Lemma \ref{lem:big_open_set_surjective}):
\begin{equation*}
\tau^* V_0 \to \underbrace{\tau^* a_* \omega_X \cong b_* \omega_W}_{\textrm{flat base-change}}\to b_* \nu _*  \nu ^*\omega_W.
\end{equation*}
As $\tau^* V_0$ is a trivial vector bundle and $h^0(b_*\nu _*  \nu ^* \omega_W)=h^0(\nu ^* \omega_W ) \leq 1$, the image of the above map, has to be of rank $1$. So, $\deg b = 1$, and hence also $\deg a=1$.

\end{proof}

\begin{corollary}
\label{cor:ordinary}
If $X$ is a smooth, projective variety with $\kappa_S(X)=0$ and $b_1(X)=2 \dim X$, then the Albanese variety $A$ of $X$ is ordinary.
\end{corollary}

\begin{proof}
From Theorems \ref{t-surj} and \ref{t-K} it follows that $a : X \to A$ is birational. However, $\kappa_S$ is a birational invariant for smooth projective varieties. Hence, $\kappa_S(A)=0$, which in fact is equivalent to $A$ being ordinary \cite[Prop. 2.3.2]{HP13}.
\end{proof}

\end{document}